\documentclass[preliminary]{eptcs}
\providecommand{\event}{ACT 2020} 
\usepackage{underscore}           

\usepackage{tikz}
\usepackage{tikz-cd}
\usetikzlibrary{decorations.pathmorphing,shapes,arrows}
\tikzset{node distance=2cm, auto}
\usetikzlibrary{positioning}
\usepackage{enumitem}
\usepackage{wrapfig}
\usepackage{sidecap}
\usepackage{caption}
\usepackage{amsmath,amssymb,amsthm}
\usepackage{pdfsync}
\usepackage{mathtools}
\usepackage{graphicx}
\usepackage{xcolor}
\usepackage{aliascnt}

\DeclareRobustCommand{\coprod}{\mathop{\text{\fakecoprod}}}
\newcommand{\fakecoprod}{%
  \sbox0{$\prod$}%
  \smash{\raisebox{\dimexpr.9625\depth-\dp0}{\scalebox{1}[-1]{$\prod$}}}%
  \vphantom{$\prod$}%
}

\newenvironment{note}{\textcolor{blue}}{}

\newtheoremstyle{myremark} 
    {7pt}                    
    {7pt}                    
    {}  	                 
    {}                           
    {\bf}       	         
    {.}                          
    {.5em}                       
    {}  

\theoremstyle{plain}
\newtheorem{lemma}{Lemma}[section]
\newaliascnt{theorem}{lemma}
\newaliascnt{definition}{lemma}
\newaliascnt{corollary}{lemma}
\newaliascnt{proposition}{lemma}
\newaliascnt{conjecture}{lemma}
\newaliascnt{claim}{lemma}
\newaliascnt{question}{lemma}
\newaliascnt{remark}{lemma}
\newaliascnt{example}{lemma}

\newtheorem{theorem}[lemma]{Theorem}
\aliascntresetthe{theorem}
\newtheorem*{theorem-main}{Theorem~\ref{thm:main}}
\newtheorem{definition}[definition]{Definition}
\aliascntresetthe{definition}
\newtheorem{corollary}[corollary]{Corollary}
\aliascntresetthe{corollary}
\newtheorem{proposition}[proposition]{Proposition}
\aliascntresetthe{proposition}
\newtheorem{conjecture}[conjecture]{Conjecture}
\aliascntresetthe{conjecture}

\aliascntresetthe{claim}
\theoremstyle{definition}
\newtheorem{question}[question]{Question}
\aliascntresetthe{question}

\theoremstyle{myremark}
\newtheorem{remark}[remark]{Remark}
\aliascntresetthe{remark}
\newtheorem{example}[example]{Example}
\aliascntresetthe{example}

\newcommand{\diam}{\mathrm{diam}}
\newcommand{\comma}[2]{(#1 \mathbin{\downarrow} #2)}
\newcommand{\rcomma}[2]{[#1 \mathbin{\downarrow} #2]}
\newcommand{\I}{\mathcal{I}}        
\newcommand{\cU}{\mathcal{U}}       
\newcommand{\vrminf}[2]{\mathrm{VR}_{\infty}^{\mathrm{m}}(#1;#2)}
\newcommand{\cechminf}[2]{\mathrm{\check{C}ech}_{\infty}^{\mathrm{m}}(#1;#2)}
\newcommand{\vrless}[2]{\mathrm{VR}_<(#1;#2)}
\newcommand{\vrmless}[2]{\mathrm{VR}_<^{\mathrm{m}}(#1;#2)}
\newcommand{\vrmleq}[2]{\mathrm{VR}_{\le}^{\mathrm{m}}(#1;#2)}
\newcommand{\km}{\mathrm{K}^{\mathrm{m}}}
\newcommand{\finkm}{\mathrm{finK}^{\mathrm{m}}}
\newcommand{\m}[1]{\mathcal{#1}}   
\newcommand{\mt}[1]{{#1}^{\mathrm{m}}}   
\newcommand{\mtm}[1]{\mt{\m{#1}}}   
\newcommand{\tx}{\tilde{x}}
\newcommand{\ts}{\tilde{s}}
\newcommand{\deltam}[1]{\delta(#1)} 
\newcommand{\smt}{\cat{MetTh}}    
\newcommand{\psmt}{\cat{pMetTh}}    
\newcommand{\Top}{\cat{Top}}
\newcommand{\cost}{\mathrm{cost}}
\newcommand{\avg}{\m{A}}
\newcommand{\bc}{\Phi}
\newcommand{\bcs}{\widetilde{\Phi}}
\newcommand{\mRX}{R_{\m{X}}}
\newcommand{\mRY}{R_{\m{Y}}}

\newcommand{\iso}{\cong}			

\newcommand{\homt}{\simeq}			

\newcommand{\inj}{\hookrightarrow}		
\newcommand{\surj}{\twoheadrightarrow} 	

\newcommand{\Z}{\mathbb{Z}}			
\newcommand{\R}{\mathbb{R}}			
\newcommand{\N}{\mathbb{N}}			

\renewcommand{\P}{\mathcal{P}}      

\newcommand{\st}{\: | \:}                       
\newcommand{\vr}[2]{\mathrm{VR}(#1;#2)}
\newcommand{\cech}[2]{\check{\mathrm{C}}\mathrm{ech}(#1;#2)}
\newcommand{\cechinf}[2]{\check{\mathrm{C}}\mathrm{ech}_\infty(#1;#2)}
\newcommand{\vrm}[2]{\mathrm{VR}^{\mathrm{m}}(#1;#2)}
\newcommand{\cechm}[2]{\check{\mathrm{C}}\mathrm{ech}^{\mathrm{m}}(#1;#2)}
\newcommand{\supp}{\mathrm{supp}}
\newcommand{\nerve}[1]{\mathcal{N}(#1)}
\newcommand{\nervem}[1]{\mathcal{N}^{\mathrm{m}}(#1)}
\newcommand{\ac}{\ll}
\renewcommand{\S}[2]{\mathcal{S}(#1,#2)}

\newcommand{\comp}{\circ}
\newcommand{\tild}{\widetilde}               

\newcommand{\disj}{\sqcup}

\newcommand{\id}{\mathrm{id}}                     

\newcommand{\cat}[1]{\mathsf{#1}}

\renewcommand{\hom}[1]{\textup{Hom}_{#1}}    

\newcommand{\dum}{\square}  
\newcommand{\term}{\bullet} 

\newcommand{\Set}{\mathsf{Set}}

\newcommand{\met}{\mathsf{Met}}

\newcommand{\pmet}{\mathsf{pMet}}
\newcommand{\scpx}{\mathsf{sCpx}}

\newcommand{\defn}[1]{\textbf{#1}}
\newcommand{\defeq}{\coloneqq} 
\newcommand{\dd}{\mathrm{d}}

\title{Operations on Metric Thickenings}
\author{
Henry Adams \qquad\qquad Johnathan Bush \qquad\qquad Joshua Mirth
\institute{Colorado State University\\
Colorado, USA}
\email{\{lastname\}@math.colostate.edu }
}
\def\titlerunning{Operations on Metric Thickenings}
\def\authorrunning{H.H.\ Adams, J.E.\ Bush \& J.R.\ Mirth}
\begin{document}
\maketitle

\begin{abstract}
Many simplicial complexes arising in practice have an associated metric space structure on the vertex set but not on the complex, e.g.\ the Vietoris--Rips complex in applied topology.
We formalize a remedy by introducing a category of simplicial metric thickenings whose objects have a natural realization as metric spaces.
The properties of this category allow us to prove that, for a large class of thickenings including Vietoris--Rips and \v{C}ech thickenings, the product of metric thickenings is homotopy equivalent to the metric thickenings of product spaces, and similarly for wedge sums.
\end{abstract}

\section{Introduction}\label{sec:introduction}

Applied topology studies geometric complexes such as the Vietoris--Rips and \v{C}ech simplicial complexes.
These are constructed out of metric spaces by combining nearby points into simplices.
We observe that proofs of statements related to the topology of Vietoris--Rips and \v{C}ech simplicial complexes often contain a considerable amount of overlap, even between the different conventions within each case (for example, $\le$ versus $<$).
We attempt to abstract away from the particularities of these constructions and consider instead a type of simplicial metric thickening object.
Along these lines, we give a natural categorical setting for so-called \defn{simplicial metric thickenings}~\cite{AdamaszekAdamsFrick2018}.

In Sections~\ref{sec:motivation} and~\ref{sec:relatedWork}, we provide motivation and briefly summarize related work.
Then, in Section~\ref{sec:simplicialMetricThickening}, we introduce the definition of our main objects of study: the category $\smt$ of simplicial metric thickenings and the associated metric realization functor $\mt{\dum}$ from $\smt$ to the category of metric spaces.
We define $\smt$ as a particular instance of a comma category and prove that this definition satisfies certain desirable properties, e.g.\ it possesses all finite products.
We define simplicial metric thickenings as the image of the metric realization of $\smt$.
Particular examples of interest include the Vietoris--Rips and \v{C}ech simplicial thickenings.

In Section~\ref{sec:limits}, we prove that certain (co)limits are preserved, up to homotopy equivalence, by the functors defined in Section~\ref{sec:simplicialMetricThickening}. 
For example, we show that the metric realization functor factors over products and wedge sums. 
We also prove that the analogous (co)limits are preserved for the Vietoris--Rips and \v{C}ech simplicial thickenings.

\section{Motivation}\label{sec:motivation}

Our motivation is twofold: first to give a general and categorical definition of simplicial metric thickenings, which first appeared in~\cite{AdamaszekAdamsFrick2018}, though primarily in the special case of the Vietoris--Rips metric thickenings.
Second, to use this framework to give succinct proofs about the homotopy types of these objects while de-emphasizing the particular details of the Vietoris--Rips or \v{C}ech complex constructions.

Let us first explain the reason to consider an alternative to the simplicial complex topology.
While the vertex set of a Vietoris--Rips or \v{C}ech complex is a metric space, the simplicial complex itself may not be.
A simplicial complex is metrizable if and only if it is locally finite, meaning each vertex is contained in only a finite number of simplices, and a Vietoris--Rips complex (with positive scale parameter) is not locally finite if it is not constructed from a discrete metric space.
Similarly, the inclusion of a metric space, $X$, into its Vietoris--Rips or \v{C}ech complex is not continuous unless $X$ is discrete, since the restriction of the simplicial complex topology to the vertex set is the discrete topology.
Both of these problems are addressed by the Vietoris--Rips and \v{C}ech \emph{metric thickenings}, which are metric spaces and for which there is a canonical isometric embedding of the underlying metric space.

As an example, let us consider in more detail the differences between the Vietoris--Rips simplicial complex and the Vietoris--Rips metric thickening at the level of objects and morphisms.
Given a metric space $X$, the Vietoris--Rips simplicial complex $\vr{X}{r}$ has as its simplicies all finite subsets $\sigma\subseteq X$ of diameter at most $r$.
We interpret this construction as an element of the image of a bifunctor $\vr{\dum}{\dum}$ with domain $\met\times [0,\infty)$, where $[0,\infty)$ is the poset $([0,\infty),\le)$ viewed as a category, and with codomain $\scpx$, the category of simplicial complexes and simplicial maps.
There is then a geometric realization functor from $\scpx$ to the category of topological spaces.
For a fixed metric space $X$, we have a functor from $[0,\infty)$ to topological spaces, in particular, a morphism $\vr{X}{r}\hookrightarrow \vr{X}{r'}$ whenever $r\le r'$.
As a simplicial complex, $\vr{X}{0}$ contains a vertex for each point of $X$ and no higher-dimensional simplices.
However, if $X$ is not a discrete metric space, then $\vr{X}{0}$ and $X$ may not even be homotopy equivalent because $\vr{X}{0}$ is the set $X$ equipped with the discrete topology.
Problems arise also for $r>0$, when $\vr{X}{r}$ need not be metrizable---a simplicial complex is metrizable if and only if it is locally finite.
In contrast, Vietoris--Rips metric thickenings are a functor $\vrm{\dum}{\dum}$ from $\met\times [0,\infty)$ to metric spaces, not just to topological spaces.
In particular, $\vrm{X}{0}$ is isometric to $X$.
Furthermore, given a 1-Lipschitz map $X\to Y$, we obtain a natural transformation $\vrm{X}{\dum}\to\vrm{Y}{\dum}$.
So, $\vrm{\dum}{\dum}$ is indeed a bifunctor from $\met\times [0,\infty)$ to $\met$.

There is a fair bit known about Vietoris--Rips complexes $\vr{X}{r}$ that does not immediately transfer to the metric thickenings $\vrm{X}{r}$.
Some properties, such as the stability of persistent homology~\cite{ChazaldeSilvaOudot2014}, are potentially difficult to transfer in a categorical fashion.
Other properties, such as statements about products and wedge sums, do transfer over cleanly.

Whereas proofs about homotopy types of Vietoris--Rips and \v{C}ech simplicial complexes often involve simplicial collapses, the corresponding proofs for metric thickenings instead often involve deformation retractions not written as a sequence of simplicial collapses.
We give two versions of this correspondence in Section~\ref{sec:limits}, including explicit formulas proving that the Vietoris--Rips thickening of an $L^\infty$ product (respectively wedge sum) of metric spaces deformation retracts onto the product (wedge sum) of the Vietoris--Rips thickenings.
Hence, thickenings behave nicely with respect to certain limits and colimits.

\section{Related Work}\label{sec:relatedWork}

This paper draws on three distinct bodies of work.
The topology of Vietoris--Rips and \v{C}ech complexes has been widely studied in the applied topology community~\cite{AdamaszekAdams2017,adamaszek2017vietoris,GasparovicCompleteCharacterization1Dimensional2017,gasparovic2018relationship,lim2020vietoris,zaremsky2019}.
Major questions include determining the homotopy type of the Vietoris--Rips complex of a given space at all scale parameters $r$ (see in particular~\cite{AdamaszekAdams2017} which determines all Vietoris--Rips complexes of the circle), and of determining the topology of the Vietoris--Rips complex of a product, wedge sum, or other gluing of spaces whose individual Vietoris--Rips complexes are known.
Metric gluings were studied extensively in~\cite{adamaszek2017vietoris}, and products in~\cite{carlsson2020persistent,gakhar2019k}.
Here we study similar questions, not about the Vietoris--Rips simplicial complex but the Vietoris--Rips metric thickening.
These latter objects were introduced in~\cite{AdamaszekAdamsFrick2018}.

A well-known construction is the metric of barycentric coordinates, which is a metrization of any simplicial complex $K$, as explained in~\cite[Section 7A.5]{bridson2011metric}, and can be considered a functor $B \colon \scpx \to \met$.
Consider a real vector space $V$ with basis $K^0$, the vertex set of $K$, equipped with an inner product $\langle - , - \rangle$ such that this basis is orthonormal.
We can realize $K$ as the set of all finite, convex, $\R$-linear combinations of basis vectors (i.e.\ vertices) contained in some simplex.
The inner product defines a metric, $d_b(u,v) = \sqrt{\langle u-v, u-v \rangle}$, on $V$.
The restriction of this metric to $K$ is called the metric of barycentric coordinates.
Dowker proves in~\cite{dowker1952topology} that the identity map from a simplicial complex $K$ with the simplicial complex topology to $K$ with the metric of barycentric coordinates is a homotopy equivalence.
A key difference between the simplicial metric thickenings considered in this paper and the metric of barycentric coordinates is the following: with barycentric coordinates (as with the simplicial complex topology) the vertex set is equipped with the discrete topology, but in a simplicial metric thickening the vertex set need not be discrete.
Another functor from simplicial complexes to metric spaces is studied in~\cite{Marin2017Measure,Marin2017}.
This functor also produces a space with the same (weak) homotopy type as the geometric realization.
Roughly, this construction is to take a simplicial complex $K$ and consider the space of random variables $X \colon \Omega \to K^0$ where $\Omega$ is some reference probability space and $K^0$ denotes the vertex set of $K$.
The space $L(\Omega,K)$ which metrizes $K$ is the subset of random variables which give positive probability to all subsets of $K^0$ which correspond to simplices in $K$, and the metric is given by the measure (in $\Omega$) of the set on which two random variables differ.
This construction also places the discrete topology on the vertex set $K^0$, and therefore typically disagrees with the homotopy type of the simplicial metric thickening.

Finally, we draw some inspiration from the idea of a probability monad in applied category theory.
A probability monad, or more specifically the Kantorovich monad~\cite{FritzPerrone2019,Perrone2018}, is a way to put probability theory on a categorical footing.
A probability monad $P$ is defined so that if $X$ is a metric space, then $PX$ is a collection of random elements in $X$.
As the main data of the monad, there is an evaluation map $PPX \mapsto PX$ defined by averaging.
Furthermore, an algebra of the probability monad, i.e.\ an evaluation map $PX \mapsto X$, is analogous to a Karcher or Fr\'{e}chet mean map as used in the proof of~\cite[Theorem~4.2]{AdamaszekAdamsFrick2018} and~\cite[Theorem~4.6, Theorem~5.5]{AdamsMirth2019}.
Moreover, the Kantorovich monad of~\cite{FritzPerrone2019} places the Wasserstein metric on the space of probability measures, as we do when defining simplicial metric thickenings.

\section{The Category of Simplicial Metric Thickenings}\label{sec:simplicialMetricThickening}

We begin by fixing some notation.
Given a metric space $X$, let $\P{X}$ denote the set of all Radon probability measures on $X$ with finite $p$-th moment.
With the $p$-Wasserstein metric, $\P{X}$ is a metric space; for details see Section~\ref{ssec:metricRealizationFunctor}.
There is a canonical inclusion $\delta \colon X \to \P{X}$ given by $\delta(x) = \delta_x$.
To avoid a proliferation of subscripts we will also write $\deltam{x}$.
We will write $\nu \ac \mu$ to mean that $\nu$ is absolutely continuous with respect to $\mu$, that is, if whenever $\mu(E)=0$ for some measurable $E \subseteq X$, then $\nu(E) = 0$.
Let $\I{X}$ denote the subspace of $\P{X}$ consisting of measures with finite support, i.e.\ those of the form $\mu = \sum_{i=0}^n \lambda_i \deltam{x_i}$.

\begin{definition}\label{def:simplicialMetricThickening}
A \defn{simplicial metric thickening} of a metric space $X$ is a subspace $\m{K}$ of $\I{X}$ which satisfies:
\begin{enumerate}[topsep=0pt,itemsep=-1ex,partopsep=1ex,parsep=1ex]
\item The image of $\delta \colon X \to \I{X}$ is contained in $\m{K}$, and
\item If $\mu \in \m{K}$ and $\nu \ac \mu$, then $\nu \in \m{K}$.
\end{enumerate}
\end{definition}

As a point of comparison, recall the definition of an abstract simplicial complex:
\begin{definition}\label{def:abstractSimplicialComplex}
An \defn{abstract simplicial complex} on a set $V$ is a subset $K$ of $2^V$ consisting only of finite sets which satisfies
\begin{enumerate}[topsep=0pt,itemsep=-1ex,partopsep=1ex,parsep=1ex]
\item The image of the map $v \mapsto \{v\}$ is in $K$, and 
\item If $\sigma \in K$ and $\tau \subseteq \sigma$, then $\tau \in K$.
\end{enumerate}
\end{definition}

\begin{example}
A motivating example of a simplicial thickening is the Vietoris--Rips simplicial metric thickening~\cite{AdamaszekAdamsFrick2018}.
Recall the Vietoris--Rips complex, $\vr{X}{r}$, as described in Section~\ref{sec:motivation}.
By necessity of the construction, the vertices of $\vr{X}{r}$ have an associated metric, even though $\vr{X}{r}$ need not be metrizable.
The associated simplicial metric thickening, $\vrm{X}{r}$, is the subset of $\P{X}$ containing all measures whose support set is a simplex in $\vr{X}{r}$, and it is a metric space.

We will frequently return to this example.
In particular, the homotopy type of the Vietoris--Rips complex of various spaces is widely studied~\cite{AdamaszekAdams2017,AdamaszekAdamsFrick2018,chacholski2020homotopical,ChazaldeSilvaOudot2014,GasparovicCompleteCharacterization1Dimensional2017,gasparovic2018relationship,Virk1DimensionalIntrinsicPersistence2017,virk2019rips,zaremsky2019}.
By formulating a category of simplicial metric thickenings which includes Vietoris--Rips thickenings, we are able to compute the homotopy type of Vietoris--Rips thickenings of spaces constructed from limit and colimit operations.
\end{example}

There are several reasonable choices of morphisms between simplicial metric thickenings.
Since they are metric spaces, any map of metric spaces could be allowed (see Section~\ref{ssec:simplicialThickeningComma} for several choices of maps of metric spaces).
Alternatively, one could define a morphism between simplicial metric thickenings $\m{K}$ and $\m{L}$ of metric spaces $X$ and $Y$, respectively, to be a function $f \colon X \to Y$ such that the pushforward $f_\# \colon \m{K} \to \P{Y}$ has its image contained in $\m{L}$.
In Sections~\ref{ssec:commaCategories} and~\ref{ssec:simplicialThickeningComma} we construct a description of a category which has as objects the simplicial metric thickenings of Definition~\ref{def:simplicialMetricThickening} and for which this latter definition of morphisms arises naturally.

\subsection{Comma Categories}\label{ssec:commaCategories}

We work with the standard notions of category theory; for further details, we refer the reader to~\cite{Riehl16} (for example).
We often abuse notation and write $c \in \cat{C}$ when $c$ is an object of the category $\cat{C}$.

\noindent
\begin{minipage}{.8\textwidth}
\begin{definition}\label{def:commaCategory}
Given functors $S \colon \cat{A} \to \cat{C}$ and $T \colon \cat{B} \to \cat{C}$, the \defn{comma category} $\comma{S}{T}$ has as objects all triples $(a,b,\phi)$ where $a \in \cat{A}$, $b \in \cat{B}$, and $\phi \colon Sa \to Tb$, and as morphisms all pairs $(f,g)$ with $f \in \hom{\cat{A}}(a,a')$ and $g \in \hom{\cat{B}}(b,b')$, such that the following diagram commutes:
\end{definition}
\end{minipage}
\begin{minipage}{.2\textwidth}
\centering
\tikzset{node distance=1.5cm}
\begin{tikzpicture}
   \node (Sa) {$Sa$};
   \node (Saa) [below of=Sa] {$Sa'$};
   \node (Tb) [right of=Sa] {$Tb$};
   \node (Tbb) [below of=Tb] {$Tb'$};
   \draw[->] (Sa) to node  {$\phi$} (Tb);
   \draw[->] (Sa) to node [swap] {$Sf$} (Saa);
   \draw[->] (Tb) to node [swap] {$Tg$} (Tbb);
   \draw[->] (Saa) to node [swap] {$\phi'$} (Tbb);
\end{tikzpicture}
\end{minipage}

We introduce the following subcategory of a comma category.
\begin{definition}\label{def:restrictedCommaCategory}
The \defn{restricted comma category} $\rcomma{S}{T}$ is the full subcategory defined to contain all objects $(a,b,\phi)\in\comma{S}{T}$ such that $\phi$ is an isomorphism.
\end{definition}

For an arbitrary comma category, the order of the source functor $S$ and target functor $T$ is important: $\comma{S}{T}$ and $\comma{T}{S}$ are not equivalent as categories in general.
However, restricted comma categories are less particular.

\begin{proposition}\label{prop:symmetry}
The categories $\rcomma{S}{T}$ and $\rcomma{T}{S}$ are isomorphic.
\end{proposition}


The proof of Proposition~\ref{prop:symmetry} is omitted. 
Our main theorems in Section~\ref{sec:limits} are about various types of limits and colimits.
As we explain below, restricted comma categories inherit these structures from their source and target categories.

Observe that any comma category $\comma{S}{T}$ has two functors $P_\cat{A} \colon \comma{S}{T} \to \cat{A}$ and $P_\cat{B} \colon \comma{S}{T} \to \cat{B}$, the \defn{domain} and \defn{codomain} functors.
These are given by sending a triple $(a,b,\phi)$ to $a$ and to $b$, respectively, and by sending a morphism $(f,g)$ to $f$ and to $g$, respectively.
We will denote the functors $P_\cat{A} \colon \rcomma{S}{T} \to \cat{A}$ and $P_\cat{B} \colon \rcomma{S}{T} \to \cat{B}$ with the same symbols.

\begin{lemma}\label{lem:limits}
Fix categories $\cat{A}$, $\cat{B}$, and $\cat{C}$ and functors $S\colon \cat{A}\to\cat{C}$ and $T\colon\cat{B}\to\cat{C}$. 
For some small index category $\cat{J}$, suppose that $\cat{A}$ and $\cat{B}$ admit colimits under $\cat{J}$-shaped diagrams and that $S$ preserves colimits under $\cat{J}$-shaped diagrams.
Then $\comma{S}{T}$ admits colimits under $\cat{J}$-shaped diagrams.

Dually, if $\cat{A}$ and $\cat{B}$ admit limits over $\cat{J}$-shaped diagrams and $T$ preserves limits over $\cat{J}$-shaped diagrams, 
then $\comma{S}{T}$ admits limits over $\cat{J}$-shaped diagrams.

\end{lemma}

\begin{proof}
We will prove only the case for colimits; the case for limits follows by dualizing the proof. 

Let $D \colon \cat{J} \to \comma{S}{T}$ be a diagram in the comma category, and denote the objects in its image by $(a_j,b_j,\phi_j)$ for $j \in \cat{J}$.
Then $P_{\cat{A}}D \colon \cat{J} \to \cat{A}$ and $P_{\cat{B}}D \colon \cat{J} \to \cat{B}$ are $\cat{J}$-shaped diagrams in $\cat{A}$ and $\cat{B}$, and so have colimits $\ell_a$ and $\ell_b$.
There is a natural transformation $\Phi = (\phi_j)_{j \in J} \colon SP_{\cat{A}}D \implies TP_{\cat{B}}D$.
Observe that $SP_{\cat{A}}D \colon \cat{J} \to \cat{C}$ is a diagram in $\cat{C}$ with colimit $S\ell_a$ because $S$ preserves colimits.
Let $Z \colon P_{\cat{B}}D \implies \ell_\cat{B}$ denote the cocone natural transformation.
Then $T Z \Phi \colon SP_{\cat{A}}D \implies T\ell_B$ is a cocone over $SP_{\cat{A}}D$, so there exists a unique morphism $\psi \colon S\ell_\cat{A} \to T\ell_\cat{B}$ (see Figure~\ref{fig:cocone}).


The colimit of $D$ is $(\ell_\cat{A}, \ell_\cat{B}, \psi)$.
Indeed, suppose that $(a,b,\chi)$ is a cocone over $D$.
Then there are unique morphisms $f_1 \colon \ell_\cat{A} \to a$ and $f_2 \colon \ell_\cat{B} \to b$ because composition with $P_\cat{A}$ or $P_\cat{B}$ gives diagrams in $\cat{A}$ and $\cat{B}$.
The morphism $(f_1,f_2) \in \hom{\comma{S}{T}}((\ell_\cat{A},\ell_\cat{B},\phi),(a,b,\chi))$ is well-defined because everything in sight commutes.
Hence, $\comma{S}{T}$ admits colimits under $J$-shaped diagrams.

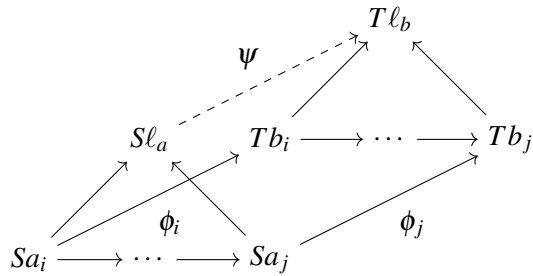
\begin{figure}[h]
\begin{center}
\tikzset{node distance=1.6cm}
\begin{tikzpicture}
    \node (ai) {$Sa_i$};
    \node (ac) [right of=ai] {$\cdots$};
    \node (aj) [right of=ac] {$Sa_j$};
    \node (la) [above of=ac] {$S\ell_a$};
    \draw[->] (ai) to node {} (ac);
    \draw[->] (ac) to node {} (aj);
    \draw[->] (ai) to node {} (la);
    \draw[->] (aj) to node {} (la);
    \node (bi) [right of=la] {$Tb_i$};
    \node (bc) [right of=bi] {$\cdots$};
    \node (bj) [right of=bc] {$Tb_j$};
    \node (lb) [above of=bc] {$T\ell_b$};
    \draw[->] (bi) to node {} (bc);
    \draw[->] (bc) to node {} (bj);
    \draw[->] (bi) to node {} (lb);
    \draw[->] (bj) to node {} (lb);
    \draw[->] (ai) to node [swap] {$\phi_i$} (bi);
    \draw[->] (aj) to node [swap] {$\phi_j$} (bj);
    \draw[->,dashed] (la) to node {$\psi$} (lb);
\end{tikzpicture}
\caption{There exists a map, $\psi$, because $T\ell_{\cat{B}}$ is a cone over $SP_{\cat{A}}D$.}
\label{fig:cocone}
\end{center}
\end{figure}
\end{proof}

\begin{corollary}\label{cor:restrictedlimits}
With the setup of Section~\ref{sec:limits}, suppose that the image of $D \colon \cat{J} \to \comma{S}{T}$ is contained in the subcategory $\rcomma{S}{T}$.
Then the limit over (respectively, colimit under) $D$ is contained in $\rcomma{S}{T}$.
\end{corollary}
\begin{proof}
In the special case that $D \colon \cat{J} \to \rcomma{S}{T}$ is a diagram in the restricted comma category, the natural transformations $\Phi$ has an inverse, $\Phi^{-1} = (\phi_j^{-1})_{j \in \cat{J}}$.
It follows that $S\ell_\cat{A}$ is a colimit under $TP_{\cat{B}}D$ and $T\ell_\cat{B}$ is a colimit under $SP_{\cat{A}}D$.
Therefore, there are unique morphisms $\psi \colon S\ell_\cat{A} \to T\ell_\cat{B}$ and $\xi \colon T\ell_\cat{B} \to S\ell_\cat{A}$ and these are necessarily isomorphisms.
Hence, $\rcomma{S}{T}$ admits colimits under $J$-shaped diagrams. 
\end{proof}

\begin{lemma}\label{lem:commaAdjoints}
Let $P_\cat{A}$ and $P_\cat{B}$ be the domain and codomain functors from $\comma{S}{T}$ to $\cat{A}$ and $\cat{B}$, respectively.
If $T$ has a left adjoint then so does $P_\cat{A}$, and if $S$ has a right adjoint so does $P_\cat{B}$.
\end{lemma}

\begin{proof}
To begin, we assume that $T$ has a left adjoint $L$, with counit $\epsilon \colon LT \implies \id_{\cat{B}}$ and unit $\eta \colon \id_{\cat{C}} \implies TL$ (see~\cite[Section 4.2]{Riehl16}, for example).
Define $\tild{L} \colon \cat{A} \to \comma{S}{T}$ by
\begin{align*}
\cat{A} \ni a &\mapsto (a, LSa, \eta_{Sa}) \in \comma{S}{T} \\
\hom{\cat{A}}(a_1,a_2) \ni f &\mapsto (f, LSf ) \in \hom{\comma{S}{T}}(\tild{L}a_1,\tild{L}a_2).
\end{align*}

We claim that $\tild{L}$ is left adjoint to $P_\cat{A}$. 
Observe that $P_\cat{A}\tild{L} = \id_{\cat{A}}$ so there is trivially a unit $\tild{\eta} \colon \id_\cat{A} \implies P_\cat{A}\tild{L}$.
We need to construct a counit $\tild{\epsilon} \colon \tild{L}P_\cat{A} \implies \id_{\comma{S}{T}}$.
Define $\tild{\epsilon}_{(a,b,\phi)} = (\id_a, \epsilon_b \comp L\phi)$, and observe that the triangle identities in Figure~\ref{fig:tennisball} (middle and right) are satisfied for this definition of counit.
\begin{figure}[h]
\hspace{10mm}
\begin{tikzpicture}
    \node (rcomma) {$\comma{S}{T}$};
    \node (A) [below=0.9cm of rcomma] {$\cat{A}$};
    \node (B) [right=1.5cm of rcomma] {$\cat{B}$};
    \node (C) [below=0.945cm of B] {$\cat{C}$};
    \draw[->] (B) [bend left] to node {$T$} (C);
    \draw[->] (A) to node [swap] {$S$} (C);
    \draw[->] (rcomma) [bend left] to node {$P_\cat{A}$} (A);
    \draw[->] (rcomma) to node {$P_\cat{B}$} (B);
    \draw[->] (C) [bend left] to node {$L$} (B);
    \draw[->,dashed] (A) [bend left] to node {$\tild{L}$} (rcomma);
\end{tikzpicture}
\hspace{15mm}
\begin{tikzpicture}
\node[scale=1.1] (a) at (0,0){
\begin{tikzcd}
\tild{L} \arrow[r, Rightarrow, "\tild{L}\tild{\eta}"] \arrow[dr, Rightarrow, "\id_{\tild{L}}"']
& \tild{L}P_\cat{A}\tild{L} \arrow[d, Rightarrow, "\tild{\epsilon}\tild{L}"]\\
& \tild{L}
\end{tikzcd}
\hspace{15mm}
\begin{tikzcd}
P_\cat{A} \arrow[r, Rightarrow, "\tild{\eta}P_\cat{A}"] \arrow[dr, Rightarrow, "\id_{P_\cat{A}}"']
& P_\cat{A}\tild{L}P_\cat{A} \arrow[d, Rightarrow, "P_\cat{A}\tild{\epsilon}"]\\
& P_\cat{A}
\end{tikzcd}
};
\end{tikzpicture}
\caption{(Left) The setup of Lemma~\ref{lem:commaAdjoints}.
(Middle and Right) Triangle identities for an adjunction.}
\label{fig:tennisball}
\end{figure}
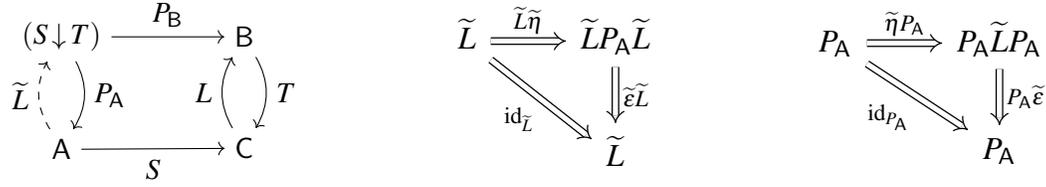

In particular, the triangle identity of Figure~\ref{fig:tennisball} (middle) is satisfied because 
\[\tild{\epsilon}_{\tild{L}a}\circ\tild{L}\tild{\eta}_a=\tild{\epsilon}_{\tild{L}a}\circ\id_{\tild{L}a}=\tild{\epsilon}_{(a,LSa,\eta_{Sa})}=(\id_a,\epsilon_{LSa}\circ L \eta_{Sa})=(\id_a,\id_{LSa})=\id_{\tild{L}a}\]
for all $a\in\cat{A}$, and the triangle identity Figure~\ref{fig:tennisball} (right) is satisfied because 
\[P_\cat{A}\tild{\epsilon}_c\circ \tild{\eta}_{P_\cat{A}c}=P_\cat{A}\tild{\epsilon}_c\circ\id_{P_\cat{A}c}=P_\cat{A} (\id_a,\epsilon_b\circ L\phi)=\id_a=\id_{P_\cat{A}c}\]
for all $c=(a,b,\phi)\in\comma{S}{T}$.

A similar argument shows that if $S$ has a right adjoint $R$ then $P_\cat{B}$ has a right adjoint $\tild{R}$.
\end{proof}

\begin{corollary}\label{cor:rcommaAdjoints}
Let $P_\cat{A}$ and $P_\cat{B}$ be the domain and codomain functors from $\rcomma{S}{T}$ to $\cat{A}$ and $\cat{B}$, respectively.
If $S$ has a left or right adjoint, then so does $P_\cat{B}$, and likewise if $T$ has a left or right adjoint, so does $P_\cat{A}$.
\end{corollary}

\begin{proof}
Apply Lemma~\ref{lem:commaAdjoints} to see that if $S$ has a right adjoint, then so does $P_B$, and that if $T$ has a left adjoint, then so does $P_\cat{A}$.
The case of when $S$ has a left adjoint or $T$ has a right adjoint follows after first applying Proposition~\ref{prop:symmetry}.
\end{proof}

\subsection{Simplicial Thickenings as a Comma Category}\label{ssec:simplicialThickeningComma}

To formalize simplicial metric thickenings as comma categories, we first recall the definitions of the categories of simplicial complexes and of metric spaces.

\begin{definition}\label{def:simplicialMap}
Let $K$ and $L$ be simplicial complexes with vertex sets $K^0$ and $L^0$.
A \defn{simplicial map} is a function $f \colon K^0 \to L^0$ such that if $\sigma$ is a simplex of $K$, then $f(\sigma)$ is a simplex of $L$.
\end{definition}

The category of simplicial complexes, $\scpx$, has abstract simplicial complexes as objects and simplicial maps as morphisms.
This category admits finite products and coproducts.
The categorical product of simplicial complexes $K$ and $L$, denoted $K \prod L$, is the simplicial complex such that $(\sigma,\tau) \in K \prod L$ is a simplex whenever $\sigma \in K$ and $\tau \in L$~\cite[Definition~4.25]{Kozlov2008}.
The coproduct, denoted $K \coprod L$, is the disjoint union simplicial complex.

\begin{definition}\label{def:shortFunction}
Let $X$ and $Y$ be metric spaces and $k \in [0,+\infty)$.
A function $f \colon X \to Y$ is \defn{$k$-Lipschitz} if $d(f(x),f(x')) \le k d(x,x')$ for all $x,x' \in X$.
Functions which are $1$-Lipschitz may be called \defn{short}.
\end{definition}

Lipschitz functions are, of course, continuous.
We define the category of metric spaces, $\met$, to have metric spaces as objects and short maps as morphisms.
While this is a standard definition (it is the same used in~\cite{FritzPerrone2019}, for example), there are alternative definitions in the literature, where either the morphisms are less-restricted, or the axioms of a metric space are relaxed~\cite{Lawver1979}.
In particular, the morphisms may be allowed to be maps which are $k$-Lipschitz for some $k \in [0,\infty)$, or simply continuous maps.
The latter is the structure of the category of metric spaces as a full subcategory of $\Top$.
Many of our constructions do not depend on the choice of morphisms for $\met$, but our default choice in this paper is short maps.

The metric space axioms may also be relaxed when defining a category of metric spaces.
Recall that the classical definition of a metric space is a set $X$ equipped with a function $d(\cdot,\cdot) \colon X \times X \to [0,\infty)$ such that 
$d(x,y) = 0$ if and only if $x=y$,
$d(x,y) = d(y,x)$ for all $x,y \in X$, and
$d(x,z) \le d(x,y) + d(y,z)$ for all $x,y,z \in X$.
Allowing $d(x,y) = \infty$ gives an \defn{extended metric space}. 
Allowing $d(x,y) = 0$ when $x \neq y$ gives a \defn{pseudo-metric space}.
Allowing $d(x,y) \neq d(y,x)$ is a \defn{quasi-metric space}.
Combining all of the three above relaxations gives Lawvere metric spaces, or categories enriched in the monoidal poset $([0,+\infty],\le,+)$.

We will make use of classical metric spaces and of extended pseudo-metric spaces, denoting the category of the latter by $\pmet$.
Of course, $\met$ is a full subcategory of $\pmet$.

The category $\met$ has finite products.
If $X$ and $Y$ are metric spaces, the product $X \times Y$ is the cartesian product of the underlying sets with the supremum norm:
$d((x,y),(x',y')) = \max\{d(x,x'),d(y,y')\}$.
Coproducts do not exist in $\met$; however, colimits under certain other diagrams do, including the wedge sum discussed in Section~\ref{sec:gluings}.


One advantage of $\pmet$ is the existence of arbitrary products and coproducts.
The product is defined 
using the supremum metric, and the coproduct $X \coprod Y$ is the set $X \disj Y$ with $d(x,y) = +\infty$ for $x \in X$ and $y \in Y$ (all other distances are unchanged).
The necessity of working in $\pmet$ for arbitrary products is shown by the following example (see~\cite[Chapter~2, Example~1.9]{manes2012algebraic} for a formal proof that arbitrary products may not exist in $\met$).
Consider the space $X = \R^{\N}$ (that is, sequences of real numbers) with the supremum norm.
The distance between $x = (0,0,\ldots,0,\ldots)$ and $y = (0,1,2,\ldots,n,\ldots)$ is then $d(x,y) = \sup_{n \in \N} n = \infty$.
All the other axioms of the metric are still satisfied by supremums taken over infinite sets, however.

Note that both the categories of metric spaces and simplicial complexes possess canonical functors to $\Set$.
For metric spaces, the functor $U$ is given by forgetting the metric $d$,
\begin{align*}
U \colon \met \ni (X,d) &\mapsto X \in \Set \\
f \colon (X,d_X) \to (Y,d_Y) &\mapsto f \colon X \to Y
\end{align*}
For abstract simplicial complexes, the functor $\dum^0$ is given by forgetting the subset structure, 
\begin{align*}
\dum^0 \colon \scpx \ni K&\mapsto K^0 \in \Set \\
f \colon K \to L &\mapsto f|_{K^0} \colon K^0 \to L^0
\end{align*}
Here $K^0$ and $L^0$ are the vertex sets of the simplicial complexes $K$ and $L$.
We will often not refer to $U$ and $\dum^0$ explicitly and instead write $X$ or $K^0$ to refer to the underlying sets.
\\

\noindent
\begin{minipage}{.8\textwidth}
\begin{definition}\label{def:categorySimplicialMetricThickenings}
The category $\smt$ of simplicial metric thickenings is the restricted comma category $\rcomma{U}{\dum^0}$.
Explicitly, objects are triples $(X,K,\phi)$, in which $X$ is a metric space, $K$ is an abstract simplicial complex, and $\phi \colon K^0 \to X$ is an isomorphism of sets, and a morphism between $(X,K,\phi)$ and $(Y,L,\psi)$ is a pair of short maps $(f \colon X \to Y, g \colon K \to L)$ such that the following diagram commutes in $\Set$\emph{:}
\end{definition}
\end{minipage}
\begin{minipage}{.2\textwidth}
\centering
\tikzset{node distance=1.5cm}
\begin{tikzpicture}
   \node (X) {$X$};
   \node (K) [below of=X] {$K^0$};
   \node (Y) [right of=X] {$Y$};
   \node (L) [below of=Y] {$L^0$};
   \draw[->] (X) to node [swap] {$\phi$} (K);
   \draw[->] (X) to node {$f$} (Y);
   \draw[->] (K) to node [swap]  {$g|_{K^0}$} (L);
   \draw[->] (Y) to node {$\psi$} (L);
\end{tikzpicture}
\end{minipage}
\\

Note that the source category of $U$ can be either $\met$ or $\pmet$, to distinguish we use $\smt$ and $\psmt$.
Next, we establish some basic properties of the category of simplicial metric thickenings.

\begin{proposition}\label{prop:basics}
The domain and codomain functors $P_\pmet \colon \psmt \to \pmet$ and $P_\scpx \colon \psmt \to \scpx$ both have left and right adjoints.
In addition, the functor $P_\met$ also defines a functor $\smt \to \met$ with left and right adjoints.
\end{proposition}

\begin{proof}
As per Corollary~\ref{cor:rcommaAdjoints}, we only need to show that $U$ and $\dum^0$ have adjoints.
Starting with $\dum^0$, the right adjoint is the complete simplicial complex functor, $C$, and the left adjoint is the trivial complex functor, $T$.

Let $D_r \colon \Set \to \pmet$ be the functor giving every set the discrete metric where all distances are equal to $r$.
The right adjoint of $U$ is $D_0$ and the left adjoint is $D_\infty$.
These are not defined for $\met$, and so $P_\scpx$ has adjoints only in $\psmt$, and not in $\smt$.
\end{proof}

Note that $\scpx$ and $\met$ can both be embedded into $\smt$.
Choosing some $r$, the functor $D_r \colon \scpx \to \smt$ is a full and faithful embedding, and the functors $T \colon \met \to \smt$ and $C \colon \met \to \smt$ are full and faithful embeddings.

\begin{proposition}\label{prop:smtlimits}
If $\pmet$ and $\scpx$ each admit (co)limits over small diagrams of shape $\cat{J}$, then so does $\psmt$.
If $\met$ and $\scpx$ each posses limits over small diagrams of shape $\cat{J}$, then so does $\smt$.
In particular, $\psmt$ admits finite products and coproducts, and $\smt$ admits finite products.
\end{proposition}

\begin{proof}
As described in Proposition~\ref{prop:basics}, $\dum^0$ and $U$ both have left and right adjoints.
Therefore, both are continuous and cocontinuous functors, i.e., they preserve small limits and colimits.
By Corollary~\ref{cor:restrictedlimits}, $\rcomma{U}{\dum^0}$ has limits of any small diagram for which limits exist in both $\met$ and $\scpx$.
\end{proof}

\subsection{The Metric Realization Functor}\label{ssec:metricRealizationFunctor}

Here we show that every object of $\psmt$ can be realized as a space satisfying Definition~\ref{def:simplicialMetricThickening}.
We call this realization the \defn{metric realization} of the simplicial metric thickening.
It was first introduced in~\cite{AdamaszekAdamsFrick2018} and is related to~\cite{FritzPerrone2019}.
Much like the convention for geometric realizations of a simplicial complexes, we will often not distinguish between an object of $\smt$ and its metric realization. 

As a point of comparison, there is a functor $|\dum| \colon \scpx \to \Top$ that takes a simplicial complex $K$ to a topological space $|K|$ called the geometric realization.
While simplicial thickenings could be given a topology using $|\dum|$ and factoring through $P_\scpx$, the metric realization functor provides a more direct topological realization with better properties due to the metric structure.
As described in~\cite{AdamaszekAdamsFrick2018}, the metric thickening of a simplicial thickening $(X,K,\phi)$ in which $K$ is locally-finite is always homeomorphic to the geometric realization of the simplicial complex $|K|$.
However, geometric realizations of non-locally-finite complexes are non-metrizable, so the metric thickening topology is necessarily different.

To define the metric realization, we need a certain number of measure-theoretic definitions.
If $X$ is a metric space, we will consider it a measurable space with its Borel $\sigma$-algebra.
Given a point $x \in X$, let $\deltam{x}$ denote the delta distribution with mass one centered at the point $x$.
By a probability measure on $X$ we mean a Radon measure $\mu$ such that $\mu(X) = 1$.
We will furthermore assume that probability measures have finite moments, meaning that for any fixed $x' \in X$ and $p \in [1,+\infty)$, we have
$\int_X d_X(x,x')^p \dd\mu(x) < \infty$.
Note that any measure with finite support and total mass one is a probability Radon measure with finite moments.
Denote the set of all probability Radon measures with finite moments on $X$ by $\P{X}$.
Recall also that the support of a measure $\mu$ is the (closed) set
\[\supp(\mu) = \{ x \in X \st \mu(A) > 0 \text{ for all open } A \ni x \} .\]

The technical restrictions on the measures in $\P{X}$ are necessary for $\P{X}$ to be a metric space under the Wasserstein (also Kantorovich or earth-movers) distance.

\begin{definition}\label{def:WassersteinMetric}
Let $X$ be a metric space and let $\mu$, $\nu$ be probability measures on $X$.
Let $\Gamma(\mu,\nu)$ be the set of all measures $\pi$ on $X \times X$ such that $\pi(X,A) = \nu(A)$ and $\pi(B,X) = \mu(B)$ for all measurable sets $A$ and $B$ (that is, all measures whose marginals are $\mu$ and $\nu$).
The \defn{$p$-Wasserstein distance} between $\mu$ and $\nu$ is 
\[ d(\mu,\nu) \defeq \inf_{\pi \in \Gamma(\mu,\nu)} \left( \int_{X \times X} d_X(x,y)^p \dd\pi \right)^{1/p} .\]

\end{definition}
For more details on the Wasserstein distance, including the fact that it defines a metric on $\P X$ and that all choices of $p$ are topologically equivalent, see~\cite{AdamaszekAdamsFrick2018,Edwards2011,Kellerer1984,Kellerer1982,Villani2003,Villani2008}.

We now have the requisite machinery to define the metric realization of a simplicial metric thickening.

\begin{definition}\label{def:metricRealizationFunctor}
The \defn{metric realization functor} $\mt{\dum} \colon \psmt \to \pmet$ is specified by the following data:
\begin{itemize}[topsep=0pt,itemsep=-1ex,partopsep=1ex,parsep=1ex]
\item For each simplicial thickening $\m{K} = (X,K,\phi)$ in $\psmt$, let $\mt{\m{K}}$ be the sub-metric space of $\P{X}$ of all probability measures $\mu$ such that $\phi(\supp(\mu)) = \sigma$ for some $\sigma \in K$.
\item For each morphism $(f,g) \colon (X,K,\phi) \to (Y,L,\psi)$, let $\mt{(f,g)}$ be the morphism taking $\mu = \sum_{i=1}^n \lambda_i \deltam{x_i}$ to $\mt{f}(\mu) = \sum_{i=1}^n \lambda_i \deltam{f(x_i)}$.
\end{itemize}
\end{definition}
Note that this also restricts to a functor $\mt{\dum} \colon \smt \to \met$.
There is no difficulty in allowing pseudo-metric spaces here, even though many references only treat classical metric spaces.
If $X$ contains some point $x'$ with $d(x',x) = \infty$ for some $x$ (and hence all $y$ within finite distance of that $x$), then no measure with $x,x' \in\supp(\mu)$ is in $\P{X}$ due to the finite moments condition.
Pseudo-metric spaces also have a natural topology and a well-defined Borel $\sigma$-algebra, so $\P{X}$ is defined for such spaces.

The objects here are precisely those described by Definition~\ref{def:simplicialMetricThickening}.
Indeed, for finitely-supported measures, we have $\nu \ac \mu$ if and only if $\supp(\nu) \subseteq \supp(\mu)$.
Therefore the morphisms are precisely functions $f \colon X \to Y$ between metric spaces such that the pushforward map $f_\# \colon \mt{\m{K}} \to \P{Y}$ has its image contained in $\mt{\m{L}}$.
This holds for any of the variants of categories of metric spaces described in Section~\ref{ssec:simplicialThickeningComma}, though in the following we always take $\met$ or $\pmet$ with short maps as morphisms.

As described earlier, the Vietoris--Rips complex provides a natural example of the construction of simplicial thickenings.
The above definitions allow us to describe the Vietoris--Rips complex as a functor:
\begin{definition}\label{def:VietorisRipsFunctor}
Let $r \in [0,+\infty]$.
The \defn{Vietoris--Rips functor} $\vr{\dum}{r} \colon \met \to \smt$ is defined by
\begin{align*}
\vr{\dum}{r} \colon \met \ni X &\mapsto (X,\vr{X}{r},\id) \\
f \colon X \to Y &\mapsto (f,f)
\end{align*}
\end{definition}
This is well-defined because $f$ is a short map and therefore sends any simplex $\sigma$ to a set of points with no larger diameter.
The \defn{Vietoris--Rips simplicial thickening} is the composition of functors $\vrm{\dum}{r}$.

A related construction is the \v{C}ech complex functor.
In a metric space $X$, we let $B_r(x)$ denote the ball of radius $r$ centered at the point $x\in X$.
\begin{definition}\label{def:CechComplex}
Let $r \in [0,+\infty]$ and let $X$ be a set.
The \defn{\v{C}ech complex}, $\cech{X}{r}$, has a simplex for every finite subset $\sigma \subseteq X$ such that $\cap_{x \in \sigma} B_r(x) \neq \emptyset$.
The \defn{\v{C}ech functor} $\cech{\dum}{r} \colon \met \to \smt$ is defined by
\begin{align*}
\cech{\dum}{r} \colon \met \ni X &\mapsto (X,\cech{X}{r},\id) \\
f \colon X \to Y &\mapsto (f,f)
\end{align*}
\end{definition}
Again, the \defn{\v{C}ech simplicial thickening} is the composition $\cechm{\dum}{r}$.
We will study both of these constructions further in Section~\ref{sec:limits}.

\section{Metric Thickenings and Limit Operations}\label{sec:limits}

Vietoris--Rips and \v{C}ech simplicial complexes preserve certain homotopy properties under products and wedge sums.
Indeed, the case of ($L^\infty$) products is given in~\cite[Proposition~10.2]{AdamaszekAdams2017},~\cite{gakhar2019k},~\cite{lim2020vietoris} and the case of wedge sums is given in~\cite{adamaszek2017vietoris,adamaszek2020homotopy,chacholski2020homotopical,lesnick2020quantifying}.

In this section we give categorical proofs for metric thickenings.
We have seen that if $\met$ and $\scpx$ have (co)limits of a certain shape, then so does $\smt$.
We now prove that certain (co)limits are preserved by the metric thickening functors $\mt{\dum}$, $\vrm{\dum}{r}$, and $\cechm{\dum}{r}$, at least up to homotopy type.

\subsection{Metric Thickenings of Products}

We begin with the product operation.
The deformation retraction we construct corresponds to the map sending a measure on a product space to the product measure of its corresponding marginals.

We use $\times$ to denote the product in $\met$ and $\scpx$, and $\prod$ for the product in $\smt$.
Since products exist in both $\met$ and $\scpx$, they exist in $\smt$ by Proposition~\ref{prop:smtlimits}.
Explicitly, the product of $\m{M} = (X,K,\phi)$ and $\m{N} = (Y,L,\psi)$ is $\m{M} \prod \m{N} = (X\times Y, K \times L, \phi \times \psi)$.

\begin{proposition}\label{prop:product}
For any simplicial metric thickenings $M$ and $N$, the metric realization factors over the product up to homotopy: 
\[\mtm{M} \times \mtm{N} \homt \mtm{\left(M \textstyle{\prod} N \right)} .\]
\end{proposition}

\begin{proof}
Let $\m{M} = (X,K,\phi)$ and $\m{N} = (Y,L,\psi)$.
Elements of $\mtm{M}$ are finitely-supported measures of the form $\mu = \sum_i \lambda_i \deltam{x_i}$ with $x_i \in X$ and $\phi(\supp(\mu))\in K$.
Likewise elements of $\mtm{N}$ have the form $\nu = \sum_j \eta_j\deltam{y_j}$ with $y_j \in Y$ and $\psi(\supp(\nu))\in L$.
Thus elements of $\mtm{M} \times \mtm{N}$ are pairs $(\mu,\nu) = (\sum_i \lambda_i \deltam{x_i} , \sum_j \eta_j\deltam{y_j})$ with $\phi(\supp(\mu))\times\psi(\supp(\nu))\in K\times L$.
On the other hand, elements of $\mtm{\left(M \prod N \right)}$ are measures on $X \times Y$ of the form $\sum_{i,j} \alpha_{i,j} \deltam{(x_i,y_j)}$.

With this in mind, there is is an obvious injection $\iota \colon \mtm{M} \times \mtm{N} \inj \mtm{\left(M \prod N \right)}$ via 
\[ \left( \sum_i \lambda_i \deltam{x_i} \, , \sum_j \eta_j\deltam{y_j} \right) \mapsto \sum_{i,j} \lambda_i \eta_j\deltam{(x_i,y_j)} .\]
Concretely, $\iota$ sends a pair of measures on $X$ and $Y$ to their product measure on $X \times Y$.
There is also a surjection $\rho \colon  \mtm{\left(M \prod N \right)} \surj \mtm{M} \times \mtm{N}$ given by taking the marginals of the joint distribution:
\[ \sum_{i,j} \alpha_{i,j} \deltam{(x_i,y_j)} \mapsto \left( \sum_i {\textstyle \left( \sum\limits_j \alpha_{i,j} \right) } \deltam{x_i} \, , \sum_{j} {\textstyle \left( \sum\limits_i \alpha_{i,j} \right) \deltam{y_j} } \right) . \]

We now show that $\iota$ and $\rho$ are homotopy inverses.
Certainly $\rho \comp \iota = \id$ by construction.
Note that the composition $\iota \comp \rho$ gives the map
\[ \sum_{i,j} \alpha_{i,j} \deltam{(x_i,y_j)} \mapsto \sum_{i,j} \left( \sum_i \alpha_{i,j} \right)\left( \sum_j \alpha_{i,j}\right) \deltam{(x_i,y_j)} . \]
This is homotopic to the identity on $\mtm{\left(M \prod N \right)}$ via the straight-line homotopy $H \colon \mtm{\left(M \prod N \right)} \times I \to  \mtm{\left(M \prod N \right)}$ where
$H(t,\mu) = t\id(\mu) + (1-t)\iota \comp \rho(\mu)$.
This is clearly well-defined as a map to $\P(X \times Y)$.
To see that the image of $H$ is in $\mtm{\left(M \prod N \right)}$, note that $\supp(\iota \comp \rho(\mu)) \subseteq \supp(\mu)$, so the entire homotopy takes place within a simplex of $K\times L$.
It then follows from~\cite[Lemma~3.9]{AdamaszekAdamsFrick2018} that homotopy $H$ is continuous.
\end{proof}

\begin{corollary}\label{cor:vr-product}
For any metric spaces $X$ and $Y$, the product operation factors through the metric Vietoris--Rips and \v{C}ech thickenings up to homotopy:
\begin{align*}
\vrm{X \times Y}{r} &\homt \vrm{X}{r} \times \vrm{Y}{r}\\
\cech{X \times Y}{r} &\homt \cech{X}{r} \times \cech{Y}{r}.
\end{align*}
\end{corollary}

\begin{proof}
As simplicial complexes, we have an isomorphism $\vr{X \times Y}{r} \iso \vr{X}{r} \prod \vr{Y}{r}$ since with the $L^\infty$ metric, a subset of $X\times Y$ has diameter equal to the maximum of the diameters of its coordinate projections.
Similarly, we have an isomorphism $\cech{X\times Y}{r} \iso \cech{X}{r}\prod\cech{Y}{r}$ of \v{C}ech simplicial complexes since a collection of $L^\infty$ balls intersect if and only if their projections onto both factors intersect.
Thus $\vrm{X \times Y}{r} \iso \mt{\left(\vr{X}{r} \prod \vr{Y}{r}\right)}$ and $\cechm{X \times Y}{r} \iso \mt{\left(\cech{X}{r} \prod \cech{Y}{r}\right)}$.
The result then follows from Proposition~\ref{prop:product}.
\end{proof}

\begin{proposition}\label{prop:disjunion}
The metric thickening functors $\mt{\dum}$, $\vrm{\dum}{r}$, and $\cechm{\dum}{r}$ all preserve coproducts.
\end{proposition}
\begin{proof}
We are working in the category $\pmet$ of pseudo-metric spaces, where coproducts exist.
Recall the coproduct $X \coprod Y$ has $d(x,y) = +\infty$ for $x \in X$ and $y \in Y$.
Hence the simplicial metric thickenings $\mt{\dum}$, $\vrm{\dum}{r}$, and $\cechm{\dum}{r}$ of a coproduct are simply the coproducts of the thickenings.
\end{proof}

\subsection{Metric Thickenings of Gluings}\label{sec:gluings}

Though Proposition~\ref{prop:disjunion} is somewhat uninteresting, another colimit operation to consider is the wedge sum.

\noindent
\begin{minipage}{.8\textwidth}\begin{definition}\label{def:wedgeSum}
Let $\term$ be the terminal object in a category $\cat{C}$.
Given $A,B \in \cat{C}$, $\term_A \colon \term \to A$, and $\term_B \colon \term \to B$, the \defn{wedge sum} of $A$ and $B$, denoted $A\vee B$, is the pushout of $\term_A$ and $\term_B$:
\\

\end{definition}
\end{minipage}
\begin{minipage}{.2\textwidth}
\centering
\tikzset{node distance=1.5cm}
\begin{tikzpicture}
   \node (term) {$\term$};
   \node (A) [below of=term] {$A$};
   \node (B) [right of=term] {$B$};
   \node (AvB) [below of=B] {$A\vee B$};
   \draw[->] (term) to node  {$\term_A$} (A);
   \draw[->] (term) to node {$\term_B$} (B);
   \draw[->] (A) to node [swap] {$\iota_A$} (AvB);
   \draw[->] (B) to node {$\iota_B$} (AvB);
\end{tikzpicture}
\end{minipage}

\vspace{-\baselineskip}\begin{proposition}
Wedge sums exist in $\met$, $\scpx$, and $\smt$.
\end{proposition}

\begin{proof}
The description of the wedge sum in each category is essentially the same.
The terminal object in $\met$ is the metric space with a single point.
The wedge sum $X \vee Y$ is the metric space
$X \disj Y / (\term_X \sim \term_Y)$,
that is, $X$ and $Y$ are ``glued together'' at the points $\term_X$ and $\term_Y$.
We will refer to this common basepoint in $X \vee Y$ as $\term$.
The metric on $X \vee Y$ is given by $d(x,y) = d(x,\term) + d(\term,y)$ for $x \in X$ and $y \in Y$, while distances within $X$ and $Y$ are unchanged.
One can check that, with this metric, $X\vee Y$ satisfies the appropriate universal property.

The terminal object in $\scpx$ is the simplicial complex with a single vertex.
The wedge sum $K \vee L$ is the simplicial complex
$K \disj L / (\term_K \sim \term_L)$,
and again we refer to the common basepoint as $\term$.

Since wedge sums exist in both $\met$ and $\scpx$, they exist in $\smt$ by Proposition~\ref{prop:smtlimits}.
The wedge sum of $\m{M} = (X,K,\phi)$ and $\m{N} = (Y,L,\psi)$ is $\m{M} \vee \m{N} = (X\vee Y, K \vee L, \phi \vee \psi)$.
\end{proof}

\begin{remark}\label{rem:coproduct}
For any simplicial metric thickenings $\m{M}$ and $\m{N}$, the metric realization factors over the wedge sum.
Indeed, we have $\mtm{M} \vee \mtm{N} = \mtm{(M \vee N)}$.
However, if $F \colon \met \to \smt$, it is too much to expect that $F(X \vee Y) \iso F(X) \vee F(Y)$.
This fails, for example, if $F$ is the Vietoris--Rips functor; see Figures~\ref{fig:minipage1} and~\ref{fig:minipage2}.
Therefore proving that the metric thickening behaves well with respect to wedge sums is more delicate than the product case.
\end{remark}

\begin{figure}[ht]
\centering
\begin{minipage}[b]{0.35\linewidth}
\includegraphics[width=\textwidth]{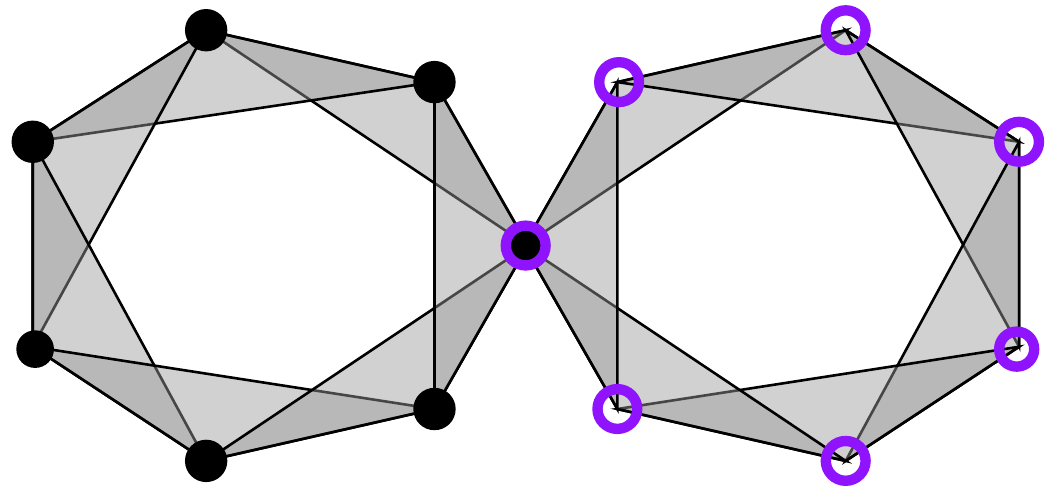}
\caption{$\vr{X}{r}\vee\vr{Y}{r}$}
\label{fig:minipage1}
\end{minipage}
\qquad
\begin{minipage}[b]{0.35\linewidth}
\includegraphics[width=\textwidth]{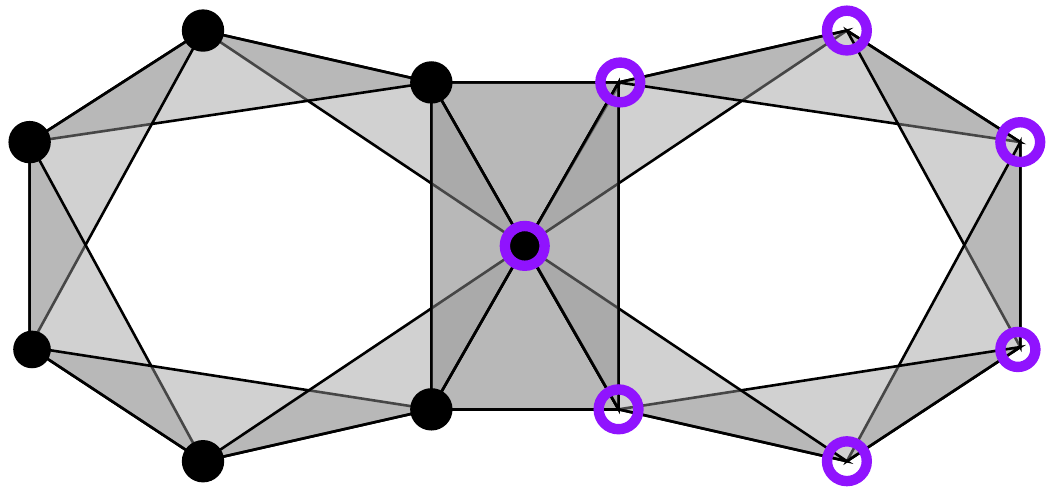}
\caption{$\vr{X \vee Y}{r}$}
\label{fig:minipage2}
\end{minipage}
\end{figure}

\begin{proposition}\label{prop:coproduct}
Let $\m{M}=(X, K, \phi)$ and $\m{N}=(Y, L, \psi)$ be simplicial thickenings.
Suppose the simplicial thickening $\m{V}=(X \vee Y, S, \phi)$ has the property that $S \supseteq K \vee L$, and if $\sigma\in S$ is a subset of neither $\phi(X)$ nor $\psi(Y)$, then $\sigma\cup \term$ is also a simplex in $S$. Then $\mtm{V} \homt \mtm{(M \vee N)}$.
\end{proposition}

\begin{proof}
Elements of both $\mtm{V}$ and $\mtm{(M \vee N)}$ have the form 
\[ \epsilon \deltam{\term}+\sum_i \lambda_i \deltam{x_i} + \sum_j \eta_j \deltam{y_j} \]
where $\epsilon + \lambda + \eta = 1$.
(Here we assume $x_i\in X$ and $y_j\in Y$, and define $\lambda = \sum_i \lambda_i$ and $\eta =\sum_j \eta_j$.)
Further, elements of $\mtm{(M \vee N)}$ must satisfy $\lambda=0$ or $\eta=0$.
Since $S \supseteq K \vee L$, there is an inclusion $\iota \colon \mtm{(M \vee N)}\inj \mtm{V}$.

Define $\rho\colon \mtm{V} \surj \mtm{(M \vee N)}$ by 
\[ \epsilon \deltam{\term}+\sum_i \lambda_i \deltam{x_i} + \sum_j \eta_j \deltam{y_j} \mapsto \begin{cases} (2\eta + \epsilon)\deltam{\term} + \left(1-\frac{\eta}{\lambda}\right)\sum_i \lambda_i \deltam{x_i} & \text{ if } \lambda\geq \eta\\
 (2\lambda +\epsilon)\deltam{\term} + \left(1-\frac{\lambda}{\eta}\right)\sum_j \eta_j \deltam{y_j} & \text{ if } \eta\geq \lambda,
\end{cases} \]
setting $\frac{\eta}{\lambda} = 1$ if $\lambda = 0$ and $\frac{\lambda}{\eta} = 1$ in the case that $\eta=0$.
To see that $\rho$ is continuous, note that the two piecewise formulas agree when $\lambda=\eta$ (in which case the image of $\rho$ is $\term$).
By construction the image of $\rho$ is in $\mtm{(M \vee V)}$, and $\rho$ is in fact a deformation retract, so $\rho\comp\iota = \id$.

To complete the proof, $\iota \comp \rho$ is homotopic to the identity via 
$H(t,\mu) = t\id(\mu) + (1-t)\iota \comp \rho(\mu)$.
Two cases are required to show that the image of $H$ is indeed $\mtm{V}$.
If $\supp(\mu)\subseteq X$ or $\supp(\mu)\subseteq Y$, then $\supp(\iota \comp \rho(\mu))=\supp(\mu)$.
Otherwise $\supp(\iota \comp \rho(\mu))=\supp(\mu) \cup \term$.
Regardless, $(\phi\vee\psi)(\supp(\mu)\cup\supp(\iota \comp \rho(\mu)))$ is a simplex in $S$ by assumption.
It then follows from~\cite[Lemma~3.9]{AdamaszekAdamsFrick2018} that the homotopy $H$ is continuous.
\end{proof}

\begin{corollary}\label{cor:vr-wedge}
For any metric spaces $X$ and $Y$, the wedge sum factors through the metric Vietoris--Rips and \v{C}ech thickenings up to homotopy:
\begin{align*}
\vrm{X \vee Y}{r} &\homt \vrm{X}{r} \vee \vrm{Y}{r}\\
\cechm{X \vee Y}{r} &\homt \cechm{X}{r} \vee \cechm{Y}{r}.
\end{align*}
\end{corollary}

\begin{proof}
The Vietoris--Rips case follows since $\vr{X\vee Y}{r}\supseteq\vr{X}{r}\vee\vr{Y}{r}$, and since if $\sigma\in\vr{X\vee Y}{r}$ is not a subset of either $X$ or $Y$, then $\sigma\cup\term\in\vr{X\vee Y}{r}$. The \v{C}ech case is analogous.
\end{proof}

We remark that in Corollary~\ref{cor:vr-wedge}, the same proof (the homotopy equivalence from Proposition~\ref{prop:coproduct}) works equally well whether $X$ and $Y$ are finite or infinite.
By contrast, proofs of analogous statements for Vietoris--Rips and \v{C}ech simplicial complexes either don't apply to the infinite setting~\cite{lesnick2020quantifying}, or alternatively need to treat the infinite setting as a separate case~\cite{adamaszek2017vietoris}.

\section{Conclusion}

We give a categorical definition for certain constructions arising in applications of topological data analysis, namely, metric thickenings of a simplicial complex. 
The utility of this approach is seen in the concise proofs and organizational schema afforded by the language of category theory. 
In particular, we introduce two equivalent definitions of the category $\smt$ of simplicial metric thickenings and prove that this category possesses a number of desirable properties, such as the existence of forgetful functors with left and right adjoints to both the category of metric spaces and the category of simplicial complexes. 
We define metric realizations of the simplicial metric thickenings in $\smt$ as images of the metric realization functor $\mt{\dum}$.
We specialize to Vietoris--Rips and \v{C}ech metric thickenings by precomposing with appropriate functors from $\met$ to $\smt$.
Furthermore, we prove that products and wedge sums factor through the resulting metric Vietoris--Rips and \v{C}ech thickenings.

We end with some open questions.
\begin{enumerate}[topsep=0pt,itemsep=-1ex,partopsep=1ex,parsep=1ex]

\item Is the stability of persistent homology afforded by Vietoris--Rips and \v{C}ech simplicial complexes~\cite{ChazaldeSilvaOudot2014} also shared by simplicial metric thickenings? See~\cite[Conjecture~6.14]{AdamaszekAdamsFrick2018}.

\item Is $\vrmless{X}{r}$ homotopy equivalent to $\vrless{X}{r}$ for any metric space $X$ and scale $r>0$, and similarly for \v{C}ech thickenings?
Here the subscript $<$ means that a finite set is included as a simplex if its diameter is strictly less than $r$.

\item If one instead allows measures of infinite support, how much does this affect the homotopy type of a simplicial metric thickening?


\end{enumerate}

\section*{Acknowledgements}
We would like to thank Alex McCleary and Amit Patel for their support of the Category Theory Lab at Colorado State University.

\nocite{*}
\bibliographystyle{eptcs}
\bibliography{generic}

\end{document}